\documentclass[12pt,a4paper]{article}
\usepackage[centertags]{amsmath}
\usepackage{amsfonts}
\usepackage{graphics}
\usepackage{graphicx}
\usepackage{amssymb}
\usepackage{amsthm}
\usepackage{listings}

\theoremstyle{plain}
\newtheorem{thm}{Theorem}[section]
\newtheorem{cor}[thm]{Corollary}
\newtheorem{lem}[thm]{Lemma}

\theoremstyle{definition}
\newtheorem{defn}{Definition}[section]

\theoremstyle{remark} \tolerance=10000 \hbadness=10000
\def \ni{\noindent}

\title{\bf The $k$- Sudoku Number of Graphs}
\author{Manju S Nair \footnote{E-mail : manjunsnair@gmail. com}\\
	Department of General Sciences\\
	Govt.  Women's Polytechnic College\\
	Ernakulam, Kerala, India.  \and Aparna Lakshmanan S. \footnote{E-mail : aparnals@cusat. ac. in, aparnaren@gmail. com}\\
	Department of Mathematics\\
	Cochin University of Science and Technology\\
	Cochin- 22, Kerala, India.  \and S Arumugam \footnote{E-mail : s. arumugam. klu@gmail. com}\\
	Department of Mathematics\\
	M S University\\
	Tirunelveli, Tamilnadu,India. 
}
\date{}
\begin{document}
	\maketitle
	\begin{abstract}
		\ni\line(1,0){360}\\
		Let $G=(V,E)$ be a graph of order $n$ with chromatic number $\chi(G)$.  Let $ k \geq \chi(G)  $ and $S \subseteq V$.  Let $ C_0 $ be a $k$-coloring of the induced subgraph $ G[S] $.  The coloring $C_0$ is called an extendable coloring, if $C_0$ can be extended to a $k$-coloring of $G$ and it is a $k$- Sudoku coloring of $G$, if $C_0$ can be uniquely extended to a $k$-coloring of $G$.   The smallest order of such an induced subgraph $G[S]$ of $G$ which admits a $k$-  Sudoku coloring is called $k$- Sudoku number of $G$ and is denoted by $sn(G,k)$.  When $k=\chi(G)$, we call $k$- Sudoku number of $G$ as Sudoku number of $G$ and is denoted by $sn(G)$.   In this paper, we have obtained the $3$- Sudoku number of some bipartite graphs $P_n$, $C_{2n}$, $K_{m,n}$, $B_{m,n}$ and $G \circ lK_1$, where $G$ is a bipartite graph and $l\geq1$.  Also, we have obtained the necessary and sufficient conditions for a bipartite graph $G$ to have $sn(G,3)$ equal to $n$, $n-1$ or $n-2$.  Also, we study the relation between $k$- Sudoku number of a graph $G$ and the Sudoku number of a supergraph $H$ of $G$. 
		\ni\line(1,0){360}\\
		\ni {\bf Keywords:} Sudoku Number, $k$- Sudoku Number\\
		
		\ni {\bf AMS Subject Classification:} {\bf Primary: 05C15, Secondary: 05C76}\\
		\ni\line(1,0){360}\\
	\end{abstract}
	\section{Introduction}
	By a graph $ G = (V,E) $, we mean a finite undirected connected simple graph of order $ n(G) = |V| $ and size $m(G) = |E| $ (if there is no ambiguity in the choice of $ G $, then we write it as $n$ and $m$, respectively).  A vertex coloring of $ G $ is a map $ f:V \rightarrow C $, where $C$ is a set of distinct colors.  It is proper if adjacent vertices of $G$ receive distinct colors of $C$, that is, if $ uv \in E(G) $, then $ f(u) \neq f(v) $.  The minimum number of colors needed for a proper vertex coloring of $G$ is called the chromatic number $ \chi(G) $, of a graph $G$. \\
	
	The concept of Sudoku coloring has been developed recently in graph coloring motivated by the famous Sudoku puzzle.  A solution to a Sudoku puzzle represents a proper vertex coloring of a graph $G$ using 9 colors with vertex set $V$ containing 81 cells where each of the 9 rows, 9 columns and nine 3 x 3 subsquares are complete graphs.  To solve the puzzle, some of the cells will be initially filled with numbers 1 to 9.  In otherwords, a proper vertex coloring $C$ of an induced subgraph $H$ of $G$ will be initially given using atmost 9 colors and that can be uniquely extended to a proper 9-coloring of $G$.  This gave the idea of Sudoku coloring of a graph $G$  \cite{maria2023sudoku}.   \\
	
	Let $ \chi(G) = k $ and $S \subset V$.  Let $ C_0 $ be a $k$-coloring of the induced subgraph $ G[S] $.  The coloring $C_0$ is called an extendable coloring, if $C_0$ can be extended to a $k$-coloring of $G$.  $C_0$ is a Sudoku coloring of $G$, if this extension is unique.  The smallest order of such an induced subgraph $G[S]$ of $G$ which admits a Sudoku coloring is called the Sudoku number of $G$ and is denoted by $sn(G)$ \cite{maria2023sudoku}. \\
	
	A bistar graph is a graph formed by joining the center vertices of two star graphs $K_{1,m}$ and $K_{1,n}$ by an edge and is denoted by $B(m,n)$.  The cycle and path on $n$ vertices are denoted by $C_n$ and $P_n$ respectively and $K_{m,n}$ represents the complete bipartite graph with partite sets of cardinality $m$ and $n$.  In particular, $K_{1,n}$ is called a star.  The corona product of two graphs $G$ and $H$, denoted by $G \circ H$, is defined as the graph obtained by taking one copy of $G$ and $|V(G)| $ copies of $H$ and joining the $i^{\text{th}}$ vertex of $G$ to every vertex in the $i^{\text{th}}$ copy of $H$.  \\
	
	Though various authors had dealt with the concept of Sudoku number under various names like critical sets, defining sets etc, only few papers are available in literature.  Interested readers may refer to  \cite{bate1999size}, \cite{cooper2014critical}, \cite{mahmoodian1997defining},  \cite{nair2024sudoku}, \cite{maria2023sudoku} and \cite{pokrovskiy2022graphs}.  In  \cite{cambie2022extremal}, Stijn Cambie has extended the notion $sn(G)$ to $k$-coloring where $ k > \chi(G) $ and denoted the new parameter by $sn(G,k)$.  Since only a few results are done on this parameter, we had done a detailed study on $sn(G,k)$ in our paper.  As the same name Sudoku number is used in  \cite{cambie2022extremal} for both the parameters $sn(G)$ and $sn(G,k)$, to avoid confusion, we call $sn(G,k)$ as $k$- Sudoku number in our paper. \\
	
	In the paper, we determined the $3$- Sudoku number of some bipartite graphs.  Also, we have obtained the necessary and sufficient conditions for a bipartite graph $G$ to have $sn(G,3)$ equal to $n$, $n-1$ or $n-2$.  Also, we study the relation between $k$- Sudoku number of a graph $G$ and the Sudoku number of a supergraph $H$ of $G$.  \\
	
	For notations and concepts not mentioned here, we refer to  \cite{balakrishnan2012textbook}\\
	\begin{defn}
		Let $G=(V,E)$ be a graph of order $n$ with chromatic number $\chi(G)$.  Let $ k \geq \chi(G)  $ and $S \subseteq V$.  Let $ C_0 $ be a $k$-coloring of the induced subgraph $ G[S] $.  The coloring $C_0$ is called an extendable coloring, if $C_0$ can be extended to a $k$-coloring of $G$ and it is a $k$- Sudoku coloring of $G$, if $C_0$ can be uniquely extended to a $k$-coloring of $G$.   The smallest order of such an induced subgraph $G[S]$ of $G$ which admits a $k$-  Sudoku coloring is called $k$- Sudoku number of $G$ and is denoted by $sn(G,k)$.  When $k=\chi(G)$, we call $k$- Sudoku number of $G$ as Sudoku number of $G$ and is denoted by $sn(G)$.  
	\end{defn}
The following lemmas and definition are useful for us. 
\begin{lem}
	Let $G$ be a graph with $\chi(G) \geq 3$.  Suppose $C_0$ is an extendable coloring of $G[S]$ for $S \subset V(G)$.  If there is a pendant vertex $v \notin S$, then $C_0$ is not a Sudoku coloring  \cite{maria2023sudoku}. 
	\label{lem1}
\end{lem}
\begin{lem}
	Let $G$ be a graph with $\chi(G) = k \geq 3$.  Suppose $C_0$ is an extendable coloring of $G[S]$ for $S \subset V(G)$.  If there is an edge $xy$ for which $x,y \notin S$ such that deg($x$) $\leq$ $k-1$ and deg($y$)$\leq$  $k-1$, then $C_0$ is not a Sudoku coloring  \cite{maria2023sudoku}. 
	\label{lem2} 
\end{lem}
\begin{lem}
	Let $L(x_i)$ be a list of colors of a vertex $x_i$ in the path $P_n = x_1x_2\dots x_n$, $1 \leq i \leq n$.  If $|L(x_i)| \geq 2$ for each $i$, then there are at least 2 list colorings of $P_n$  \cite{maria2023sudoku}. 
	\label{lem3}
\end{lem}

\begin{lem}
	Suppose there exists a list coloring of the cycle $C_n = v_1v_2 \dots v_nv_1$ such that the list of colors for each vertex $v_i$ satisfies the following conditions. 
	\begin{itemize}
		\item $|L(v_i)| \geq 2$
		\item $L(v_i) \subseteq \{1,2,3\}$
	\end{itemize} Then there are at least two list colorings of $C_n$  \cite{maria2023sudoku}. 
\label{lem4}
\end{lem}

\section{$3$- Sudoku number of some bipartite graphs}
The following theorem establishes the exact value of $3$- Sudoku number of various subclasses of bipartite graphs. 
\begin{thm}
	The $3$- Sudoku number of 
	\begin{enumerate}
		\item [(i)]the path $P_n$ is $\lceil\frac{n+1}{2} \rceil$. 
		\item [(ii)]the even cycle $C_{2n}$, where $n \geq 2$, is $n$. 
		\item [(iii)]the star graph $K_{1,n}$, where $n \geq 2$, is $n$. 
		\item [(iv)]the complete bipartite graph $K_{m,n}$, where $2 \leq m \leq n$, is $m$. 
		\item [(v)]the bistar $B_{m,n}$ is $m+n$, except for $m=n=1$ and $sn(B_{1,1},3)=3$. 	
	\end{enumerate}
\label{thm2}
\end{thm}
\begin{proof}
	\begin{enumerate}
		\item [(i)]	Let $v_1, v_2, \dots,  v_n$ denote the vertices of $P_n$.  For any initial coloring of $P_n$ using $3$ colors, the two pendant vertices of $P_n$ must be initially colored.  Also, since the degree of the remaining vertices is $2$, atleast one vertex from each edge must also be initially colored.  Hence, $sn(P_n,3) \geq \lceil \frac{n+1}{2} \rceil$. \\
	
	Now, let $C$ be the initial coloring of $P_n$ defined as follows. \\
	
	When $n$ is odd, color the vertices $v_1, v_3, \dots, v_n$ with colors $1$ and $2$ alternately.  When $n$ is even, color the vertices $v_1, v_3, \dots, v_{n-1}$ with colors $1$ and $2$ alternately and $v_n$ with color $3$. 	Then $C$ is uniquely extendable to a proper $3$-coloring of $P_n$ with $ \lceil \frac{n+1}{2} \rceil $ vertices initially colored.  Thus $sn(P_n,3)= \lceil \frac{n+1}{2} \rceil$. 

\item [(ii)]Since $k=3$ and the degree of all vertices in $C_{2n}$ is $2$, atleast one vertex from each edge in $C_{2n}$ must be initially colored.  Hence, $sn(C_{2n},3) \geq n$. \\
Let $v_1, v_2, \dots, v_{2n}$ be the vertices of $C_{2n}$ and the initial coloring $C$ be as follows. \\
When $n$ is even, color the vertices $v_1, v_3, v_5, \dots,  v_{2n-1}$ with colors $1$ and $2$ alternately.  When $n$ is odd, color the vertices $v_1, v_3, v_5, \dots,  v_{2n-3}$ with colors $1$ and $2$ alternately and $v_{2n-1}$ with color $3$.  Clearly, $C$ is uniquely extendable to a proper $3$-coloring of $C_{2n}$ and hence $sn(C_{2n},3)=n$. 
\item [(iii)]Since we have $3$ colors available, all the $n$ pendant vertices must be initially colored.  Therefore, $ sn(K_{1,n},3) \geq n $. \\
Let the initial coloring of $K_{1,n}$ be as follows.  Color all the pendant vertices using only two colors say $1$ and $2$.  Then the root vertex is forced to receive color $3$.  Hence, $sn(K_{1,n},3) = n$. 
\item [(iv)]Let $X$ and $Y$ be the partition of the vertex set of $K_{m,n}$ with cardinality $m$ and $n$ respectively and let $C$ be the initial coloring. 	Atmost two colors can be used in any partition, since otherwise, $C$ will not be extendable. 	If all the vertices of the partition $X$ are colored with one color, then every vertex of $Y$ will have two colors in their color list and for $C$ to be uniquely extendable, every vertex of $Y$ must be initially colored.  If some vertices of $X$ are colored with two different colors, then all the vertices of $Y$ forcefully receives the third color and hence the remaining vertices of $X$ will have two colors in their color list.  So for $C$ to be uniquely extendable, every vertex of $X$ must be initially colored.  Therefore in either case, one partition set must be completely colored for $C$ to be uniquely extendable.  Since cardinality of $X$ is minimum, $sn(K_{m,n},3) \geq m$. \\
Consider the initial coloring $C$ as follows. \\
Choose the partition set $X$.  Color one vertex of $X$ with color 1 and remaining vertices of $X$ with color 2.  Then $C$ is uniquely extendable to a proper 3-coloring of $K_{m,n}$ with $m$ vertices. 	Thus $sn(K_{m,n},3)=m$. 

\item [(v)]Since all the $(m+n)$ pendant vertices must be initially colored, $sn(B_{m,n},3) \geq m+n$.  Conversely, let $u$ be the root vertex of the pendant vertices $u_1, u_2, \dots, u_m$ and $v$ be the root vertex of the pendant vertices $v_1, v_2, \dots,  v_n$ where both $m$ and $n$ are greater than 1.  Let the initial coloring $C$ be as follows. \\
Color $u_1, u_2, \dots, u_m$ using two colors $1$ and $2$ and color $v_1, v_2, \dots, v_n$ using two colors say $1$ and $3$.  Then, $C$ is uniquely extendable to a proper $3$-coloring of $B_{m,n}$ with $(m+n)$ vertices.  Hence, $sn(B_{m,n},3)=m+n$. \\

If either $m$ or $n$ is equal to $1$, say $n=1$, then let the initial coloring $C$ be as follows. \\
Color $u_1, u_2, \dots,  u_m $ using two colors say $1$ and $2$ and color $v_1$ with color $1$.  Hence $sn(B_{m,1},3)=m+1$. \\

If $m=n=1$, then $B_{1,1}= P_4$.  We have already proved that $sn(P_n,3)= \lceil \frac{n+1}{2} \rceil$.  Therefore, $sn(P_4,3)= \lceil \frac{4+1}{2} \rceil = 3. $

\end{enumerate}
\end{proof}
Figure~\ref{f:P6} gives the $3$- Sudoku coloring of $P_6$  with 4 initially colored vertices and its unique extension. \\
\begin{figure}[h]
	\centering
	\includegraphics[width=14cm]{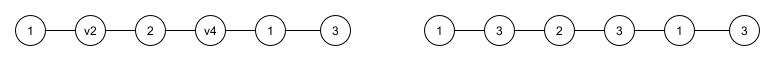}
	\caption{$3$- Sudoku coloring of $P_6$ and its final coloring}
	\label{f:P6}
\end{figure}
Figure~\ref{f:C6} gives the $3$- Sudoku coloring of $C_6$  with 3 initially colored vertices and its unique extension. \\
\begin{figure}[h]
	\centering
	\includegraphics[width=14cm]{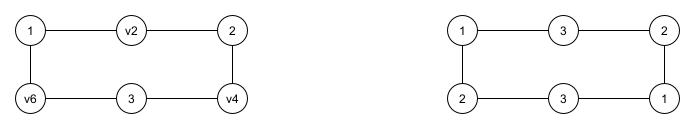}
	\caption{$3$- Sudoku coloring of $C_6$ and its final coloring}
	\label{f:C6}
\end{figure}
Figure~\ref{f:K1,6} gives the $3$- Sudoku coloring of $K_{1,6}$  with 6 initially colored vertices and its unique extension. \\
\begin{figure}[h]
	\centering
	\includegraphics[width=12cm]{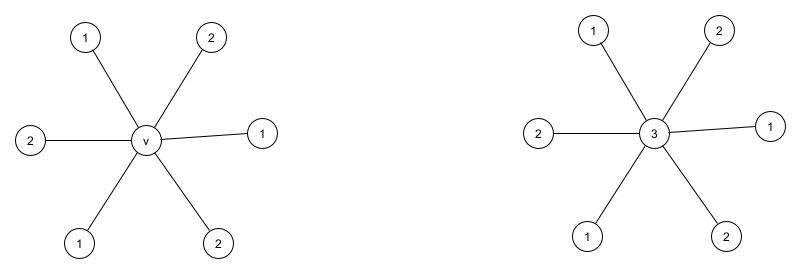}
	\caption{$3$- Sudoku coloring of $K_{1,6}$ and its final coloring}
	\label{f:K1,6}
\end{figure}
Figure~\ref{f:K3,4} gives the $3$- Sudoku coloring of $K_{3,4}$  with 3
initially colored vertices and its unique extension. \\
\begin{figure}[h]
	\centering
	\includegraphics[width=14cm]{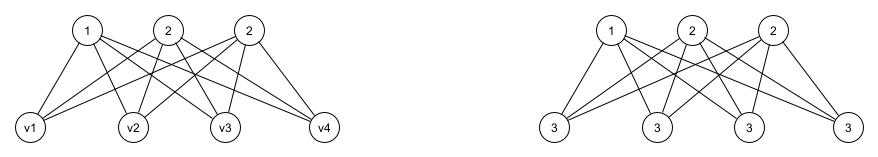}
	\caption{$3$- Sudoku coloring of $K_{3,4}$ and its final coloring}
	\label{f:K3,4}
\end{figure}
Figure~\ref{f:B3,2} gives the $3$- Sudoku coloring of $B_{3,2}$  with 5
initially colored vertices and its unique extension. \\
\begin{figure}[h]
	\centering
	\includegraphics[width=12cm]{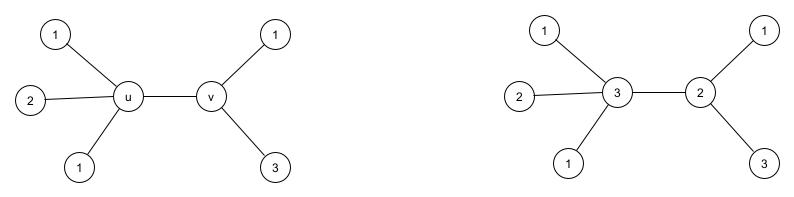}
	\caption{$3$- Sudoku coloring of $B_{3,2}$ and its final coloring}
	\label{f:B3,2}
\end{figure}
\begin{thm}
	For any bipartite graph $G$ of order $n$,
	\begin{equation*}
		sn(G \circ lK_1) = \begin{cases} n+1; & \mbox{if  $l=1$ }\\
			ln;    & \mbox{if $l \geq 2$}. 
		\end{cases}
	\end{equation*}
\end{thm}

\begin{proof}
	Case I: $l=1$ \\
	
	Since all the $n$ pendant vertices must be initially colored, $sn(G \circ K_1,3) \geq n$.  Suppose only $n$ pendant vertices are initially colored.  Then, if we consider any pair of the remaining vertices, there exists a path between them, with two colors in the color list of each vertex of that path.  Hence, the initial coloring is not uniquely extendable.  Therefore, $sn(G \circ K_1,3) \geq n+1$. \\
	
	Now, let the initial coloring $C$ be as follows. \\
	Color all the pendant vertices with color 1 and one vertex $v$ of $G$ with color 2.  Then all vertices at an odd distance from $v$ will receive color 3 and all vertices at an even distance from $v$ will receive color 2.  Hence, $C$ is uniquely extendable to a proper 3-coloring of $ G \circ K_1 $ with $ n+1 $ vertices.  Hence, $ sn(G \circ K_1, 3)=n+1 $. \\
	
	Case II: $l \geq 2$\\
	
	Since all the $ln$ pendant vertices must be initially colored, $sn(G \circ lK_1,3) \geq ln$. \\
	Now, let the initial coloring $C$ be as follows. \\
	Color one pendant vertex with color 1 and the remaining pendant vertices with color 2.  Then $C$ is uniquely extendable to a proper 3-coloring of $G \circ lK_1$ with $ln$ vertices. 
\end{proof}
Figure~\ref{f:GoK1} gives the $3$- Sudoku coloring of $G\circ K_1$  with 8
initially colored vertices and its unique extension. \\
\begin{figure}[h]
	\centering
	\includegraphics[width=14cm]{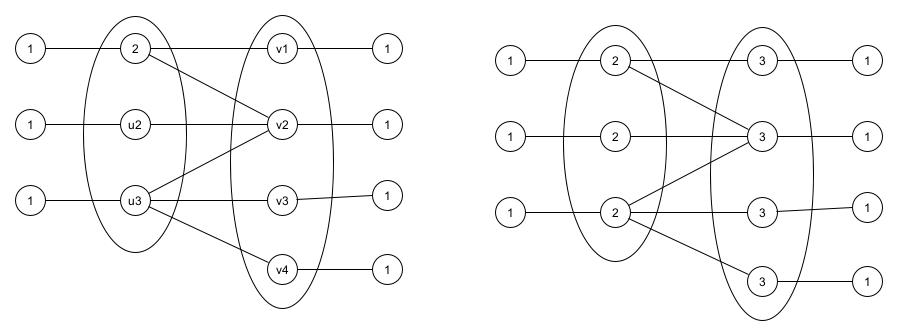}
	\caption{$3$- Sudoku coloring of $G\circ K_1$ and its final coloring}
	\label{f:GoK1}
\end{figure}
Figure~\ref{f:Go2K1} gives the $3$- Sudoku coloring of $G\circ 2K_1$  with 14
initially colored vertices and its unique extension. \\
\begin{figure}[h]
	\centering
	\includegraphics[width=14cm]{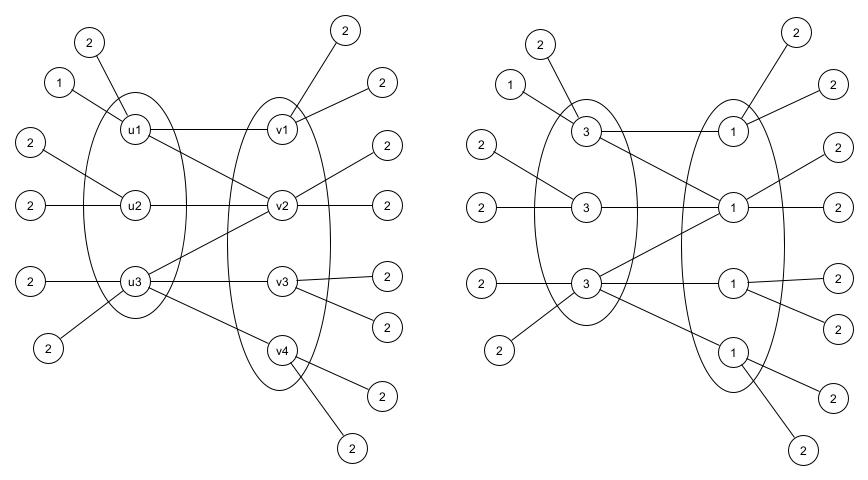}
	\caption{$3$- Sudoku coloring of $G\circ 2K_1$ and its final coloring}
	\label{f:Go2K1}
\end{figure}

\section{Some other results}
\begin{thm}
	Let $G$ be a graph of order $n$.  Then,
	$sn(G,k)=n$  if and only if $k \geq \Delta(G)+2$. 
	\label{thm1}
\end{thm}
\begin{proof}
	Assume that $k \geq \Delta(G) + 2$.  Let $C$ be any initial coloring of $G$.  The vertices of $G$ that are not yet colored can be colored using Greedy algorithm so that, in every extension step, there will be at least two colors in the color list of every vertex.  Hence, all the vertices must be initially colored so that $sn(G,k)=n$. \\
	
	Conversely, let $sn(G,k)=n$.  If possible, let $k \leq \Delta(G)+1$.  Consider a vertex $v$ with maximum degree $ \Delta(G)$.  Color all its neighbours with $k-1$ different colors and this can be extend to a proper coloring of $G-v$ using Greedy algorithm.  Now, $v$ is forced to receive a color and hence $sn(G,K) \leq n-1$, which is a contradiction.  Therefore, $k \geq \Delta(G)+2$. 
\end{proof}
\begin{thm}
	Let $G$ be a bipartite graph with $l$ pendant vertices and order $n(G)$.  Then,
\begin{equation*}
	l \leq sn(G,3) \leq n(G). 
\end{equation*}	 
Moreover, \begin{enumerate}
	\item[(i)] $ sn(G,3) = n(G)$ if and only if $ G $ is either $K_1$ or $K_2$. 
	\item[(ii)] $sn(G,3) = n(G)- 1$ if and only if $ G $ is either a star $K_{1,n}$ or a path $P_4$. 
\end{enumerate}
\end{thm}
\begin{proof}
	The upper bound is trivially true and the lower bound follows from Lemma ~\ref{lem1}. 
	\begin{enumerate}
		\item[(i)] By Theorem \ref{thm1}, $sn(G,k)=n$  if and only if $k \geq \Delta(G)+2$.  Here $ k = 3 $.  So that the inequality becomes $ 3 \geq \Delta(G)+2$.  That is, $\Delta(G) \leq 1$.  $G$ being connected, $\Delta(G)\leq 1$ implies $G$ is either $K_1$ or $K_2$.  Also, $sn(K_1,3)=1$ and $sn(K_2,3)=2$.  Hence the result. 
		\item[(ii)] By Theorem \ref{thm2}(iii), if $G$ is a star $K_{1,n}$, then $sn(G,3)=n(G)-1$.  By Theorem \ref{thm2}(i), if $G$ is a path $P_4$, then $sn(G,3)=3=n(G)-1$. \\
		
		Conversely let $sn(G,3)=n(G)-1$. \\
		
		If there are $n(G)-1$ pendant vertices, then $G$ is a star $K_{1,n}$.  If there are $n(G)-2$ pendant vertices, then $G$ is a bistar $B_{m,n}$ and by Theorem \ref{thm2}(v), $sn(G,3)=n(G)-1$ only when $G=B_{1,1}=P_4$. \\
		Suppose there are $n(G)-r$ pendant vertices where $r\geq 3$.  Since $G$ is a connected bipartite graph, the $r$ non pendant vertices also form a connected bipartite graph and will contain an induced subgraph $P_3$, say $uvw$.  Since $u$ and $w$ are non pendant vertices, they will have atleast one more neighbour, say $u^{'}$ and $w^{'}$ (need not be distinct) respectively, other than $v$. \\
		Let $C$ be an extendable coloring of $G$ as follows.  Color all vertices of $G$ other than than $u$ and $w$ in such a way that $u^{'}$ and $w^{'}$ are given a color different from that given to $v$.  Then $C$ is a uniquely extendable coloring with $n(G)-2$ initially colored vertices which is a contradiction.  Therefore, $sn(G,3)=n(G)-1$ only when $G$ is a star $K_{1,n}$ or a path $P_4$.  Hence the result. 
	\end{enumerate}
\end{proof}
\begin{thm}
	Let $G$ be a bipartite graph of order $n$.  Then, $sn(G,3)=n(G)-2$ if and only if $G$ is any one of the following graphs. 
	\begin{itemize}
		\item a bistar $B_{m,n}$ except for $m=n=1$. 
		\item a path $P_5$ or $P_6$. 
		\item a cycle $C_4$. 
		\item a cycle $C_4$ with exactly one pendant vertex attached to any one of the vertices of $C_4$. 
		\item a path $P_4$ with $l$ pendant vertices attached to one end of $P_4$ where $l>1$. 
		\item a path $P_5$ with a pendant vertex attached to the central vertex of $P_5$. 
	\end{itemize}
\end{thm}
\begin{proof}
		 By Theorem \ref{thm2}(v), if $G$ is a bistar $B_{m,n}$, then except for $m=n=1$, $sn(G,3)=m+n=n(G)-2$. \\
		 By Theorem \ref{thm2}(i), $sn(P_5,3)=3=n(G)-2$ and $sn(P_6,3)=4=n(G)-2$. \\
		 By Theorem\ref{thm2}(ii), $sn(C_4,3)=2=n(G)-2$. \\
		 Let $G$ be the graph obtained from a cycle $C_4$, say $v_1v_2v_3v_4v_1$ by attaching a pendant vertex $v_5$ to any one of the vertices of $C_4$, say $v_1$.  By Lemma \ref{lem3} and Lemma \ref{lem4}, $sn(G,3)\geq 3$.  Let $C$ be an extendable coloring of $G$ as follows.  $C(v_1)=1$, $C(v_3)=2$ and $C(v_5)=3$.  Then $C$ is uniquely extendable and hence $sn(G,3)=3=n(G)-2$. \\
		 Let $G$ be the graph obtained from a path $P_4$, say $v_1v_2v_3v_4$ by attaching $l$ pendant vertices to $v_4$ where $l>1$.  By Lemma \ref{lem2}, $sn(G,3)\geq n(G)-2$.  Let $C$ be an extendable coloring of $G$ as follows.  $C(v_1)=1, C(v_2)=2$ and color all the $l$ pendant vertices using color 1 and color 2.  Then $C$ is uniquely extendable and hence $sn(G,3)=n(G)-2$. \\
		 Let $G$ be the graph obtained from a path $P_5$, say $v_1v_2v_3v_4v_5$ by attaching a pendant vertex $v_6$ to the central vertex $v_3$.  By Lemma \ref{lem3}, $sn(G,3)\geq 4$.  Let $C$ be an extendable coloring of $G$ as follows.  $C(v_1)=C(v_5)=C(v_6)=1$ and $C(v_3)=2$.  Then $C$ is uniquely extendable and hence $sn(G,3)=4=n(G)-2$. \\
		 
		 Conversely, let $sn(G,3)=n(G)-2$. \\
		 
		 Case I: $G$ has $n(G)-2$ pendant vertices. \\
		 
		 Then $G$ is a bistar $B_{m,n}$ and by Theorem \ref{thm2}(v), $sn(G,3)=n(G)-2$ except for $m=n=1$. \\
		 
		 Case II: $G$ has $n(G)-3$ pendant vertices. \\
		 
		 Since $G$ is a connected bipartite graph, the three non pendant vertices will form an induced subgraph $P_3$, say $uvw$. \\
		 
		 Subcase I: $d(v)=2$
		 \begin{enumerate}
		 	\item If $d(u)=d(w)=2$, then $G$ is a path $P_5$ for which $sn(P_5,3)=n(G)-2$. 
		 	\item If $d(u)=2$ and $d(w)>2$ or vice-versa, then $G$ is the graph obtained from $P_4$ with $l$ pendant vertices attached to one end of $P_4$ where $l>1$ and $sn(G,3)=n(G)-2$. 
		 	\item If $d(u)>2$ and $d(w)>2$, then color the pendant vertices of $u$ with color 1 and color 2 and that of $w$ with color 2 and color 3 so that we get a uniquely extendable coloring with $n(G)-3$ initially colored vertices which is a contradiction. 
		 \end{enumerate}
	 Subcase II: $d(v)\geq3$\\
	 
	If $d(u)>2$ and $d(w) \geq 2$ or vice-versa, then color the pendant vertices of $u$ with color 1 and color 2 and that of $v$ and $w$ by color 1.  Then, we get a uniquely extendable coloring with $n(G)-3$ initially colored vertices which is a contradiction. \\
	If $d(u) = d(w) = 2$ and $d(v)>3$, then color the pendant vertices of $v$ with color 1 and color 2 and that of $u$ and $w$ by color 1.  Then, we get a uniquely extendable coloring with $n(G)-3$ initially colored vertices which is a contradiction. \\
	If $d(u)=d(w)=2$ and $d(v)=3$, then $G$ is the graph obtained from $P_5$ by attaching a pendant vertex to the central vertex of $P_5$ and $sn(G,3)=n(G)-2$. \\

	 Case III: $G$ has $n(G)-r$ pendant vertices where $r\geq 4$\\
	 
	 Since $G$ is a connected bipartite graph, the $r$ non pendant vertices also form a connected bipartite graph and will contain an induced subgraph $P_3$, say $uvw$ and atleast one more non pendant vertex, say $x$ adjacent to atleast one vertex of this $P_3$. \\
	 
	 Subcase I: $x$ adjacent to $v$\\
	 
	 Since $G$ is bipartite, $x$ cannot be adjacent to $u$ and $w$.  Since $u$, $w$ and $x$ are non pendant vertices, they will have atleast one more neighbour other than $v$.  Let an extendable coloring $C$ of $G$ be as follows.  Color all vertices of $G$ other than $u$, $w$ and $x$ in such a way that all neighbours of $u$, $w$ and $x$ are given a color different from that given to $v$.  Then $C$ can be uniquely extended with $n(G)-3 $ initially colored vertices which is a contradiction. \\
	 
	 Subcase II: $x$ adjacent to $u$ (or equivalently $w$) alone
	 \begin{enumerate}
	 	\item  Let $d(u)>2$.  Then $u$ will have atleast one more neighbour other than $x$ and $v$.  Let $C$ be an extendable coloring of $G$ as follows.  Color all vertices of $G$ other than $u$, $w$ and $x$ in such a way that all neighbours of $u$ other than $x$ and all neighbours of $w$ are given a color different from that given to $v$.  Color all neighbours of $x$ other than $u$ with color given to $v$.  Then $C$ can be uniquely extended with $n(G)-3$ initially colored vertices which is a contradiction. 
	 	\item Let $d(u)=2$, but $d(v)>2$, then a similar set of arguments work. 
	 	\item Let $d(u)=d(v)=2$\\
	 	If $d(x)=d(w)=2$, then $x$ will have one more neighbour $x_1$ and $w$ will have one more neighbour $w_1$.  If $x_1$ and $w_1$ are pendant vertices, then $G$ is a path $P_6$ for which $sn(P_6,3)=4=n(G)-2$. \\
	 	Suppose $x_1$ is a pendant vertex and $w_1$ is a non pendant vertex or vice-versa or both $x_1$ and $w_1$ are non pendant vertices.  Then $w_1$ will have atleast one more neighbour other than $w$,  say $w_2$.  Let $C$ be an extendable coloring of $G$ as follows.  $C(x_1)=C(w)=1, C(u)=C(w_2)=2$.  Color the remaining vertices of $G$ other than $x$, $v$ and $w_1$ such that $C$ is a proper coloring.  Then $C$ is uniquely extendable with $n(G)-3$ initially colored vertices which is a contradiction.  \\
	 	Let $d(x)>2$.  Let $C$ be an extendable coloring of $G$ as follows.  $C(v)=1$, $C(w_1)=2$.  Color the neighbours of $x$ using color 1 and color 2.  Color the remaining vertices of $G$ other than $x$, $u$ and $w$ such that $C$ is a proper coloring.  Then $C$ is uniquely extendable with $n(G)-3$ initially colored vertices which is a contradiction. \\
	 	If $d(x)=2$, but $d(w)>2$, then a similar set of arguments work. 
	 \end{enumerate}
 Subcase III: $x$ adjacent to both $u$ and $w$\\
 
 If $G$ has exactly four vertices, then $G$ is a $C_4$ and $sn(C_4,3)=2=n(G)-2$.  If $G$ has exactly five vertices, say $u$, $v$, $w$, $x$ and $y$, then $y$ is either adjacent to opposite pair of vertices or adjacent to any one of the vertices of $C_4$. 
 \begin{enumerate}
 	\item Let $y$ be adjacent to opposite pair of vertices, say $x$ and $v$.  Color $x$ with color 1 and $v$ with color 2 so that we get a uniquely extendable coloring with $n(G)-3$ initially colored vertices which is a contradiction. \\
 	
 	We can use similar argument for any bipartite graph with atleast six vertices which has such a graph $G$ as an induced subgraph. 
 	\item Let $y$ be adjacent to any one of the vertices of $C_4$, say $x$.  Then $sn(G,3)=n(G)-2$.  Suppose $G$ has atleast one more vertex, say $z$.  Then, $G$ will contain any one of the graphs given in Figure \ref{f:subgraphs} as an induced subgraph. (Graphs not mentioned in the figure are already considered in Case III Subcase III(1)).  If we initially color the vertices of those induced subgraphs as shown in the figure and color the remaining vertices of $G$ such that it is a proper coloring, then we get a uniquely extendable coloring with $n(G)-3$ initially colored vertices which is a contradiction.  Hence the result. 
 \end{enumerate}
\end{proof}
\begin{figure}[h]
	\centering
	\includegraphics[width=14cm]{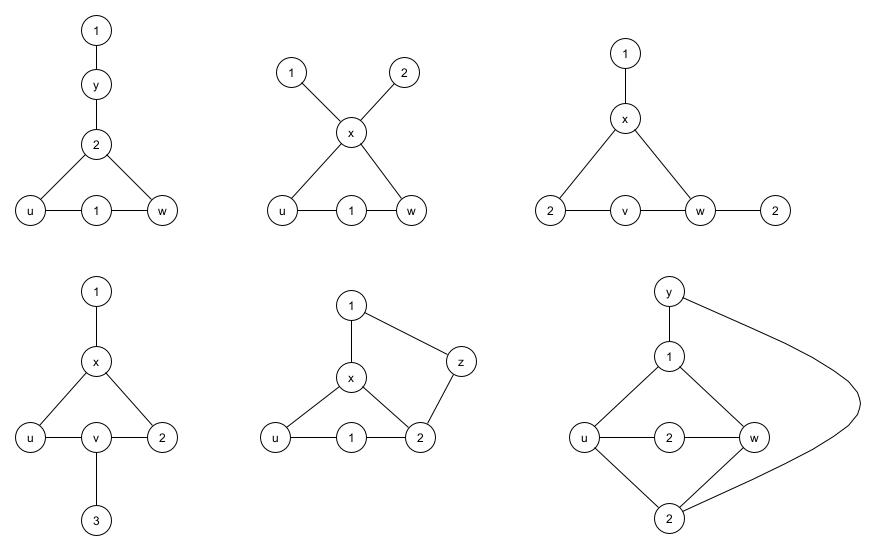}
	\caption{Induced subgraphs of $G$}
	\label{f:subgraphs}
\end{figure}

We could observe that, if the number of colors used in Sudoku coloring is greater than $\chi(G) $, there is a drastic change in the Sudoku number of a graph and hence, it becomes difficult to study its properties.  So, if we can find a supergraph $H$ for $G$, such that $\chi(H)=\chi(G)+1$ and $sn(G,\chi(G)+1)=sn(H)$, then it is enough to study the properties of $sn(H)$ rather than that of $sn(G,\chi(G)+1)$.  In our attempt to find such a graph $H$, we obtained the following results. 
\begin{thm}
	Let $H$ be the graph obtained from a path $P_n$ by attaching a complete graph $K_m$ on any edge of $P_n$.  Then,\\
	$sn(H)=sn(P_n,m)+m-4 = n+m-4$, where $m \geq 4$. \\
	Also, $sn(H)=sn(P_n,3)$. 
	\label{th2}
\end{thm}
\begin{proof}
	Let $u_1, u_2, \dots, u_n$ be the vertices of path $P_n$.  Let $H$ be the graph obtained by attaching a complete graph $K_m$ to the edge $u_iu_{i+1}$ of $P_n$ for a fixed $i \in \{1,2, \dots, n-1\}$ and let the vertices of $K_m$ are $u_i,u_{i+1},v_3,v_4,v_5, \dots v_m$. \\
	
	Case I : When $m \geq 4$,\\
 Let $C$ be an initial coloring of $H$ as follows. \\
	
	$C(v_j)=j $ for $j= 3,4, \dots, m\\$\\
	$C(u_{i-1})=1$ ; $C(u_{i+2})=2$. \\
	
	Color all the other vertices except $u_i$ and $u_{i+1}$ with colors $1$ and $2$ alternately.  Then $u_i$ is forced to receive color $2$ and $u_{i+1}$ is forced to receive color $1$.  Thus, $C$ is a uniquely extendable coloring of $H$ with $n+m-4$ initially colored vertices.  Hence, $sn(H) \leq n+m-4$. \\
	
	If possible assume that there exists an extendable coloring $C$ of $H[S]$, where $S$ is a collection of vertices of $H$ with $\lvert S \rvert \leq n+m-5$.  All the $(n-2)$ vertices of $P_n$ whose degree $\leq 2$, must be initially colored.  Otherwise, there will be at least two colors in their color list and hence $C$ will not be a Sudoku coloring.  That means, atmost $m-3$ vertices of $K_m$ are initially colored.  So at least three vertices of $K_m$ are not yet colored and hence contains at least two colors in their color list.  Therefore, at least two vertices among them are not u. c. e vertices and thus $C$ is not a Sudoku coloring.  Hence,  $sn(H)=n+m-4$. \\
	
	Case II : When $m=3$\\
	
	Let $C$ be an initial coloring of $H$ as follows. \\
	
	When $n$ is odd, color the vertices $u_1,u_3,\dots, u_n$ with color $1$ and color $2$ alternately.  When $n$ is even, color the vertices $u_1,u_3, \dots,  u_{n-1}$ with color $1$ and color $2$ alternately and $u_n$ with color $3$.  Then $C$ is a uniquely extendable coloring of $H$ and hence $sn(H) \leq \lceil\frac{n+1}{2} \rceil$. \\
	
	If possible, let $sn(H) \leq \lceil\frac{n+1}{2} \rceil -1$.  Then, by pigeonhole principle, either there exists an edge $xy$ for which $x,y \notin S$ such that $deg(x) \leq 2$ and $deg(y)\leq 2$ or the three vertices $u_i,u_{i+1},v_3$ of $K_3$ are not initially colored.  In both cases, $C$ is not a Sudoku coloring.  Hence, $sn(H)=\lceil\frac{n+1}{2} \rceil=sn(P_n,3)$. 
\end{proof}
Figure~\ref{f:P6K5} gives the Sudoku coloring of the graph $H$ obtained by attaching $K_5$ to an edge of $P_6$ with 7 initially colored vertices and its unique extension. \\
\begin{figure}[h]
	\centering
	\includegraphics[width=14cm]{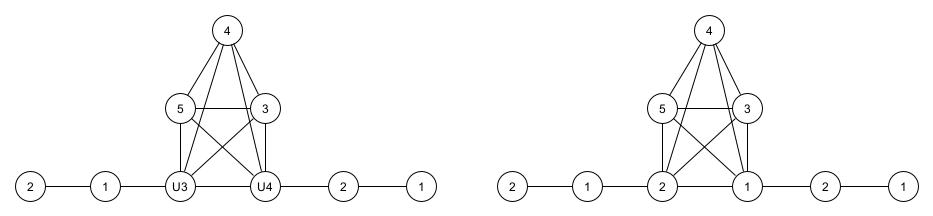}
	\caption{Sudoku coloring of the graph $H$ obtained by attaching $K_5$ to an edge of $P_6$ and its final coloring}
	\label{f:P6K5}
\end{figure}
Figure~\ref{f:P6K3} gives the Sudoku coloring of the graph $H$ obtained by attaching $K_3$ to an edge of $P_6$ with 4 initially colored vertices and its unique extension. \\
\begin{figure}[h]
	\centering
	\includegraphics[width=14cm]{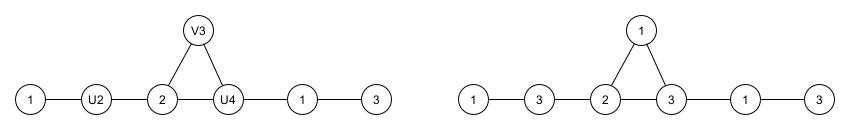}
	\caption{Sudoku coloring of the graph $H$ obtained by attaching $K_3$ to an edge of $P_6$ and its final coloring}
	\label{f:P6K3}
\end{figure}
\begin{thm}
	 Let $H$ be the graph obtained from a bipartite graph $G$ of order $n$ by attaching a complete graph $K_m$ on any edge of $G$.  Then, $sn(H) \leq sn(G,m) +m-3$. 
\end{thm}
\begin{proof}
	Let the complete graph $K_m$ be attached to the edge $uv$ of the bipartite graph $G$ of order $n$ to obtain $H$.  Since, $sn(G,m)$ is the $m$- Sudoku number of G, the initial coloring $C$ of $G$ with $sn(G,m)$ vertices will be uniquely extended to a proper $m$-coloring of $G$.  Color $m-3$ vertices of $K_m$ other than $u$ and $v$ with $m-3$ different colors other than the colors received by $u$ and $v$ in the initial coloring $C$ of $G$.  This will lead to a unique extension to proper coloring of $H$.  Hence  $sn(H)\leq sn(G,m)+m-3$. 
\end{proof}
{\bf Note:} There are graphs which satisfy strict inequality for the above result.  Examples are paths the value of which is proved as \ref{th2} above. 
\begin{thm}
	 If $ \chi(G)= \omega(G)$, and $H$ is the graph obtained from $G$ by adding a new vertex $w$ which is adjacent to all vertices of the clique of size $ \omega(G)$, then $sn(H)= sn(G,\chi(G)+1)$. 
\end{thm}
\begin{proof}
	Let $w$ be the new vertex joined to all vertices of the clique of size $ \omega (G)$.  Since, $sn(G, \chi(G)+1)$ is the $[\chi(G)+1]$- Sudoku number of $G$, the initial coloring $C$ of $G$ with $sn(G,\chi(G)+1)$ vertices will be uniquely extended to a proper $(\chi(G)+1)$ coloring of $G$.  Then $w$ is forced to receive the color other than the colors received by the vertices of the clique.  So $sn(H)\leq sn(G,\chi(G)+1)$. \\
	
	Let $C$ be an initial coloring of $H[S]$ where $S$ is a collection of vertices of $H$ with $\lvert S \rvert = sn(H)$.  If $w\notin S$, all the initially colored vertices are in $G$ and the coloring $C$ can be uniquely extended.  So $sn(G,\chi(G)+1) \leq sn(H)$. 
	If $w \in S$, then atmost $\omega(G)$ vertices of the clique are not initially colored.  Choose any one of those uncolored vertices in the graph $H-w$, color that vertex with the color given to $w$ in the initial coloring of $H[S]$.  This will lead to a unique extension to a proper $(\chi(G)+1)$-coloring of $G$.  So,  $sn(G,\chi(G)+1)\leq sn(H)$.  Hence, $sn(H) = sn(G,\chi(G)+1)$. 
	
\end{proof}
\begin{cor}
	Let $H$ be the graph obtained from a bipartite graph $G$ by adding a new vertex $w$ and joining to two adjacent vertices of $G$.  Then,
	$sn(H)=sn(G,3)$. 
\end{cor}
\begin{proof}
	Same argument as above. 
	
\end{proof}
\begin{thm}
	Given a graph $G_1$ with $\chi(G_1)=k$, there exist graphs $G_2$ and $G_3$ such that $G_i$ is an induced subgraph of $G_j$ for $i\leq j$ with $\chi(G_2)=k$, $\chi(G_3)=k+1$ and $sn(G_3)=sn(G_2,k+1)$. 
\end{thm}
\begin{proof}
	The proof is by construction. \\
	Let $G_1$ be a graph with $n$ vertices and $\chi(G_1)=k\geq3$.  Let $G_2$ be the graph $G_1\circ kK_1$ and $G_3$ be the graph obtained from $G_2$ by adding a new vertex $w$ which is adjacent to all vertices of $G_1$.  Then $\chi(G_2)=k$ and $\chi(G_3)=k+1$.  Since $k \geq 3$, all the $nk$ pendant vertices of $G_2$ and $G_3$ must be initially colored.  Hence, $sn(G_2,k+1) \geq nk$ and $sn(G_3) \geq nk$.  Let $C$ be an initial coloring of $G_2$ and $G_3$ as follows. \\
	Color the $k$ pendant vertices of each vertex $v$ of $G_1$ with $k$ colors other than the color given to $v$ in the proper $k$-coloring of $G_1$.  This will lead to a unique $k+1$-coloring of $G_2$ and $G_3$.  Hence, $sn(G_3)=sn(G_2,k+1)=nk$. 
\end{proof}
{\bf Open Question:} Given a graph $G$ with $\chi(G)=k$.  Can we embed $G$ in another graph $H$ such that $\chi(H)=k+1$ and $sn(H)=sn(G,k+1)$?\\
%\bibliography{reference}
%\bibliographystyle{acm}

\end{document}